\documentclass[12pt]{article}
\usepackage[margin = 1in]{geometry}
\usepackage{dsfont}
\usepackage{amsmath}
\usepackage{amssymb}
\usepackage{amsthm}
\usepackage{enumitem}
\usepackage{verbatim}
\usepackage{xcolor}
\newtheorem{theorem}{Theorem}[section]
\newtheorem{lemma}[theorem]{Lemma}

\theoremstyle{definition}
\newtheorem{definition}[theorem]{Definition}
\newtheorem{notation}[theorem]{Notation}

\newtheorem{remark}[theorem]{Remark}

\numberwithin{equation}{section}

\usepackage{hyperref}

\def\R{\mathbb{R}}
\newcommand{\hull}{\widetilde{\cup}}
\newcommand{\V}{\textup{Vol}}
\newcommand{\MV}{\textup{MV}}

\newcommand{\C}{\mathbb{C}}

\newcommand{\CP}{\mathbb{CP}}

\newcommand{\K}{\mathcal{K}}
\newcommand{\cS}{\mathcal{S}}

\newcommand{\bb}{\beta}

\newcommand{\om}{\omega}

\newcommand{\De}{\Delta}

\newcommand{\dd}{\partial}
\newcommand{\dbar}{\bar\partial}

\newcommand{\PSH}{\text{PSH}}

\newcommand{\abs}[1]{\lvert#1\rvert}

\def\K{\mathcal{K}}
\def\H{\mathcal{H}}

\begin{document}

\title{The Hausdorff distance and metrics on toric singularity types}

\author{A. Aitokhuehi, B. Braiman, D. Cutler, T. Darvas, R. Deaton, \\ P. Gupta, J. Horsley, V. Pidaparthy, J. Tang}

\date{}

\maketitle
\begin{abstract} Given a compact K\"ahler manifold $(X,\omega)$, due to the work of Darvas--Di Nezza--Lu, the space of singularity types of $\omega$-psh functions admits a natural pseudo-metric $d_\mathcal S$ that is complete in the presence of positive mass. When restricted to model singularity types, this pseudo-metric is a bona fide metric.

In case of the projective space, there is a known one-to-one correspondence between toric model singularity types and convex bodies inside the unit simplex. Hence in this case it is natural to compare the $d_\mathcal S$ metric to the classical Hausdorff metric. We provide precise H\"older bounds, showing that their induced topologies are the same. 

More generally, we introduce a quasi-metric $d_G$ on the space of compact convex sets inside an arbitrary convex body $G$, with $d_\mathcal S = d_G$ in case $G$ is the unit simplex. We prove optimal H\"older bounds  comparing $d_G$ with the Hausdorff metric. Our analysis shows that the H\"older exponents differ depending on the geometry of $G$, with the worst exponents in case $G$ is a polytope, and the best in case $G$ has $C^2$ boundary.
\end{abstract}

\section{Introduction}\label{sec: Introduction}
The interactions between complex and convex geometry are profound and numerous. See \cite{lazerzfeldmustataconvexandcomplex}, \cite{lehmannjiancorrespondencebetweencomplexandconvex}, \cite{Nazarovmahlerusingbergmankernel} for a few more recent instances of this, along with the recent surveys \cite{bosurvey, yanirsurvey} and references therein. In the presence of toric symmetries, complex geometric phenomenon can often be reduced to studying an analogous problem in convex analysis (see \cite{Fultonsbookontoricvarieties}). This note follows this tradition as we will analyze certain metrics introduced by  Darvas--Di Nezza--Lu in complex geometry in the toric context.

\paragraph{Pluripotential theoretic origins.} Let $(X,\om)$ be a compact K\"ahler manifold. For the basics of pluripotential theory on compact K\"ahler manifolds we refer to the excellent textbook by Guedj--Zeriahi \cite{GZbook}. With slight abuse of precision, an upper semicontinuous function $u : X \to \R\cup \{-\infty\}$ is $\om$-plurisubharmonic ($\om$-psh) if $\om_{u} := \om + i\dd\dbar u\geq 0$ in the sense of currents. We denote by $\PSH(X,\om)$ the set of all $\om$-plurisubharmonic functions. 

If $u, v \in \PSH(X,\om)$, we say $u$ is more singular than $v$, denoted by $u \preceq v$, if there exists a constant $C$ such that $u \leq v + C$. We say $u$ and $v$ have the same singularity type, denoted by $u \simeq v$, if $u\preceq v$ and $v \preceq u$. Let $[u]$ denote the equivalence class of the singularity type of $u$. Not all singularity types are created equal. To solve the complex Monge--Amp\`ere equation with prescribed singularities, in \cite{Darvasmonotonicity}  Darvas--Di Nezza--Lu introduced the so called model potentials (see \cite{darvas2023relative} for a detailed exposition). These are the potentials $\phi \in \PSH(X,\om)$ that satisfy 
\[
\phi = (\sup\{ u \in \PSH(X,\om) \, : \, u \preceq \phi, u \leq 0\})^{*}
\]
where $f^{*}(x) = \limsup_{y \to x} f(y)$ is the upper semicontinuous regularization of $f$. Slightly deviating from the literature, we denote by $\cS(X,\om)$ the set of all singularity types of model potentials $\phi \in \PSH(X,\om)$ that have positive non-pluripolar mass ($\int_X \omega^n_\phi>0$). For a detailed study of model potentials, we refer to the recent survey \cite{darvas2023relative}.

In \cite{darvas2019metric}, Darvas--Di Nezza--Lu defined a metric $d_{\cS}$ on $\cS(X,\om)$. The description of $d_{\mathcal{S}}$ involves an embedding into the space of weak Mabuchi geodesic rays, and it is quite involved. However, there is a quasi-metric equivalent to $d_{\mathcal{S}}$ which can be described in terms of non-pluripolar products of potentials. Indeed, in \cite[Proposition 3.5]{darvas2019metric}, the authors prove that if $[u], [v] \in \cS(X,\om)$, then

\begin{equation}\label{eq: interpreting dS with non-pluripolar volume}
d_{\mathcal{S}}([u], [v]) \leq \sum_{j=0}^{n} \left( 2\int_{X} \om^{j}\wedge \om^{n-j}_{\max(u,v)} - \int_{X} \om^{j}\wedge \om^{n-j}_{u} - \int_{X} \om^{j} \wedge \om^{n-j}_{v} \right) \leq C d_{\mathcal{S}}([u],[v]). 
\end{equation}
Here $C > 0$ is a constant that depends only on $n$.

For the sake of our note, it is enough to consider the particular case of   $X=\CP^{n}$ with the Fubini--Study metric $\omega=\om_{FS}$. In the coordinates $(z_{1}, \dots, z_{n})$, for the embedding of $\C^{n}$ into $\CP^{n}$ (via the map $(z_1,\ldots,z_n) \to [1,z_1,\ldots,z_n]$), the Fubini--Study metric is given by 
\[
\om_{FS} = \frac{\sqrt{-1}}{2}\sum_{j,k=1}^{n}\frac{\dd^{2}}{\dd z_{j} \dd \bar{z}_{k}} \ln (1 + |z|^{2})dz_{j} \wedge d\bar{z}_{k}.
\]

The potential $\ln(1+ |z|^{2})$ is a toric symmetric psh function and its corresponding convex function on $\mathbb{R}^n$ is $K(x) := \ln(1+e^{2x_{1}} + \dots + e^{2x_{n}})$. The closure of the subgradient set $\overline{\dd K(\R^{n})} = \De$ is the unit simplex in $\R^{n}$ given by the convex combination of $0, e_{1}, \dots, e_{n}$, where $e_{i} = (0, \dots, 1 ,\dots, 0)$. 

As we now elaborate, due to \cite[Theorem 7.2]{darvaslogconcavity}, there is a correspondence between the closed convex subsets of $\De$ and model singularity types with toric symmetries of $\om_{FS}$-psh functions. If $P \subset \De$ is a closed convex set, let $h_{P} : \R^{n} \to \R$ be the support function given by
\[
h_{P}(x) = \sup_{p \in P} \langle p, x\rangle.
\]
The corresponding psh function on $(\C^{*})^{n}$ is given by 
\[
H_{P}(z) = h_{P}(\log|z_{1}|, \dots, \log|z_{n}|). 
\]
The function 
\[
\phi_{P}(z) = H_{P}(z) - \frac{1}{2} \ln(1+ |z|^{2})
\]
defined on $(\C^{*})^{n}$ extends to all of $\CP^{n}$ as an $\om_{FS}$-psh function. As pointed out in \cite[Theorem 7.2]{darvaslogconcavity},  $\phi_{P}$ has model singularities. On the other hand, if $\phi \in \PSH(X,\om_{FS})$ has model singularities, and is toric invariant, then we can find a convex function $h$, such that $\phi(z) + \frac{1}{2} \ln (1+ |z|^{2}) = h(\log|z_{1}|, \dots, \log|z_{n}|)$. Then $P = \overline{\dd h(\R^{n})}$ is the set of all subdifferentials of $h$ and $\phi \simeq \phi_{P}$. We refer to \cite[Section 7]{darvaslogconcavity} for more details.

\paragraph{Convex geometric interpretation.}
The expression in the middle of \eqref{eq: interpreting dS with non-pluripolar volume} has an interpretation in terms of convex geometry using mixed volumes of convex bodies. Let $K,L \subset \De$ be closed convex subsets. Let $\phi_{K}, \phi_{L}$ be the corresponding $\om_{FS}$-psh potentials. 
As explained under \cite[equation (37)]{darvaslogconcavity} and \cite{Bayraktarzerodistribution}, the non-pluripolar volume can be interpreted using the classical mixed volumes $MV_j(\cdot,\cdot)$ of convex geometry (see Definition \ref{mixed_vol_def}). We observe that $\max\{h_{K}, h_{L}\} = h_{K\hull L}$ and consequently $\max\{\phi_{K}, \phi_{L}\} = \phi_{K\hull L}$ where $K\hull L$ is the convex hull of $K \cup L$. As a result, the expression in the middle of\eqref{eq: interpreting dS with non-pluripolar volume} with $u$ and $v$ replaced by $\phi_{K}$ and $\phi_{L}$ is equal to 
\begin{equation}\label{eq: equation for quasi-metrics}
d_{\De}(K,L) : = d_\mathcal S(\phi_K, \phi_L)= \sum_{j=0}^{n}\left( 2\MV_{j}(\De, K\hull L) - \MV_{j}(\De, K) - \MV_{j}(\De, L) \right).    
\end{equation}

Here $\MV_{j}(\De, K) = \MV(\De[n-j], K[j])$ is the mixed volume of $\De$ and $K$ with $K$ appearing $j$ times. 

Thus~\eqref{eq: equation for quasi-metrics} induces a quasi-metric on $\mathcal{K}(\De)$, the space of all compact convex subsets of $\De$. A quasi-metric induces a metrizable topology on $\mathcal{K}(\De)$ (see \cite{Paluszynski2009OnQA}). It raises a natural question: is this metrizable topology the same as the one induced by the familiar  Hausdorff metric $d_{\H}$?  In this paper, we answer this question affirmatively. More precisely, we prove the following result that holds more generally for closed convex subsets of a fixed convex body $G$: 

\begin{theorem}\label{thm: Main theorem}
    Let $G \subset \R^{n}$ be a convex body (compact convex set with non-empty interior). Let $\mathcal{K}(G)$ denote the set of all compact convex subsets of $G$. We consider the following expression for $K,L \in \mathcal K(G)$:
    \begin{equation}\label{eq: expression for quasi-metric for arbitrary convex body}
    d_{G} (K, L) := \sum_{j=0}^{n}\left( 2\MV_{j}(G, K\hull L) - \MV_{j}(G, K) - \MV_{j}(G, L) \right)    
    \end{equation}
    There exists $C > 1$ that depends only on $G$ such that 
    \begin{equation} \label{eq: bound for general reference body}
    \frac{1}{C}d_{\H}(K,L)^n \leq d_{G}(K,L) \leq Cd_{\H}(K,L), \ \ K,L \in \mathcal K(G).     
    \end{equation}
\end{theorem}

In the case $G = \Delta$, the estimates of \eqref{eq: bound for general reference body} imply that $(\mathcal K(\Delta),d_\Delta)$ is complete, because so is $(\mathcal K(\Delta),d_\mathcal H)$, due to the Blashke selection lemma. Remarkably, the analogous result in the K\"ahler case for general singularity types only holds under the assumption of non-vanishing mass, as pointed out in \cite[Theorem 4.9]{darvas2019metric} and the example in Section 4.2 of the same paper.

\begin{remark} \label{rem: quasi-trianle_ineq}Abusing precision, in this note we will refer to $d_{G}$ as a quasi-metric (satisfying only the quasi-triangle inequality). We will never assume that the quasi-triangle inequality for $d_G$ holds since this is only known in case $G = \De$. However in the context of the relative pluripotential theory described in \cite{darvas2023relative}, there is a natural avenue to obtain the quasi-triangle inequality for general $G$. Specifically, one can introduce a `relative metric' $d_{\mathcal{S}, \phi}$ on the space of relative singularity types $\mathcal{S}(X,\om, \phi) = \left\{[u] : u \preceq \phi, u \text{ is a model potential }\right\}$. It is natural to expect that
    \[
    d_{\mathcal{S},\phi}([u], [v]) \leq \sum_{j=0}^{n}2\int_{X} \om^{j}_{\phi}\wedge \om^{n-j}_{\max\{u,v\}} - \int_{X} \om^{j}_{\phi}\wedge \om^{n-j}_{u} - \int_{X} \om^{j}_{\phi}\wedge \om^{n-j}_{v} \leq d_{\cS, \phi}([u], [v]).
    \]
In the particular case when $\phi$ is equal to $\phi_{G}$, the potential induced by the convex body $G$, we would obtain that $d_{G}$ as described in~\eqref{eq: expression for quasi-metric for arbitrary convex body} is indeed a quasi-metric, just like $d_{\De}$. 
\end{remark}

When the reference convex body $G$ has $C^{2}$ boundary, the H\"older estimates of the previous result can be sharpened: 

\begin{theorem}\label{thm: improvement for smooth bodies}
    If $G \subset \R^{n}$ is a compact subset of $\R^{n}$ with non-empty interior and $C^2$ boundary, then there exists $C > 1$ depending only on $G$ such that 
    \begin{equation}\label{eq: bound for smooth body}
    \frac{1}{C}d_{\H} (K, L)^{\frac{n+1}{2}} \leq d_{G} (K,L) \leq C d_{\H}(K,L), \ \ K,L \in \mathcal K(G).    
    \end{equation}
    
\end{theorem}

Lastly, we argue that the H\"older estimates in the previous theorems are in a certain sense optimal:

\begin{theorem}\label{thm: bounds are optimal}Let  $G \subset \R^{n}$ be a compact subset of $\R^{n}$ with non-empty interior. For any $\bb > 1$, there is no $C > 0$ such that 
    \[
    d_{G}(K,L) \leq C d_{\H}(K,L)^{\bb}, \ \  K,L \in \mathcal K(G).
    \]
For any $p < \frac{n+1}{2}$ there is no $C$ such that 
    \[
    Cd_{\H}(K,L)^{p} \leq d_{G}(K,L)\ \  K,L \in \mathcal K(G).
    \]
Lastly, if $G$ is a convex polytope, then for any $p < n$, there is no $C> 0$ such that 
    \[
    C d_{\H}(K,L)^{p}  \leq d_{G}(K,L)\ \  K,L \in \mathcal K(G).
    \]

\end{theorem}

With regards to future directions, it would be intriguing to prove that $d_G$ satisfies a quasi-triangle inequality using only convex analysis, without any reference to complex geometry (see Remark \ref{rem: quasi-trianle_ineq}). We expect that our analysis carries through for other types of toric K\"ahler manifolds as well, not just complex projective space. Still, there is interest in seeing if any difficulties arise, and we leave these to the interested reader.

\paragraph{Organization.} In Section~\ref{sec: Background}, we explain our notations and recall relevant results from the literature. In Section~\ref{sec: holder estimates} we prove Theorems~\ref{thm: Main theorem} and~\ref{thm: improvement for smooth bodies}. In Section~\ref{sec: holder exponent optimality}, we prove Theorem~\ref{thm: bounds are optimal}. 

\paragraph{Acknowledgements.} This paper is a result of an REU project carried out in the summer of 2024. T. Darvas was the faculty mentor for the project, assisted by two graduate TAs, P. Gupta  and V. Pidaparthy. A. Aitokhuehi, B. Braiman, D. Cutler, R. Deaton, J. Horsley and J. Tang were the participating students. We acknowledge support from the NSF through the grants DMS 1846942, 2149913, 2405274.

\section{Background}\label{sec: Background} 
In this section, we introduce notation and gather some relevant definitions and theorems. 
\begin{notation}
We first recall some traditional notation. For $A, C \subset \R^n$ we say 
\begin{enumerate}[label=\arabic*.]
    \item $\mathrm{int}(A)$ is the interior of $A$ in the Euclidean topology,
    \item $\mathrm{diam}(A) :=$ $\sup \{|x-y| : x,y \in A\}$,
    \item $A\,\hull\,C$ is the closed convex hull of $A \cup C$, and
    \item $\V$ is the Lebesgue measure on $\R^n$.
\end{enumerate}

We also consider some more specific notation. For two points $p, q \in \R^n$, let $$[p,q] := \{p\} \,\hull\, \{q\} \quad \text{and} \quad [p,q) := [p,q]\setminus q.$$

By $\mathcal K^n$ we mean the space of compact, convex sets in $\R^n$. If $G \in \K^n$ has nonempty interior, we call it a \textit{convex body}. Given convex body $G$, we define $\K(G)$ to be the space of compact, convex subsets of $G$. Furthermore, $B$ is the closed unit ball centered at the origin and $\omega_k$ is the volume of the unit $k$-ball in $\R^k$. 
\end{notation}
We recall the definition of the classical Hausdorff distance.
\begin{definition}\label{def: hausdorff-dist}
    Let $K, L$ be two compact subsets of $\R^n$. The Hausdorff distance between them is 
    \[
    d_\H(K,L)=\inf\{\delta > 0 : K \subset L + \delta B, \; L \subset K + \delta B\},
    \]
    or equivalently $$d_\H(K,L) = \max \left\{\sup_{x \in K} d(x, L), \sup_{y \in L} d(y, K)\right\},$$ where $d(x,L) := \min\{|x - y| : y \in L\}.$
\end{definition}

By the Blaschke selection theorem, the Hausdorff distance turns $\mathcal K^n$ into a complete metric space with the Heine-Borel property (\cite{schneider2014}, Theorem 1.8.7). We next recall the definition of the support function of a convex set.
\begin{definition} \label{def: support-function}
    Given $K \in \mathcal K^n$, its support function $h_K : \R^n \to \R$ is 
    $$h_K(x) = \max\{u \cdot x : u \in K\}.$$
\end{definition}

Here, $u\cdot x$ denotes the Euclidean inner product between $u$ and $x$. The support $h_K$ is continuous and uniquely describes $K$. For other elementary properties of the support function, we refer the reader to Chapter 2.3 of \cite{hug2020lectures}. Additionally, we recall the definition of the mixed volume.

\begin{definition}[\cite{schneider2014}, Theorem 5.1.7]\label{mixed_vol_def}
   There exists some non-negative symmetric function $\MV: (\K^n)^n \to \R$ such that for any $K_1,\dots,K_n \in \mathcal K^n$ and $t_1,\dots,t_n \geq 0$, 
\begin{align}\label{eq: minkow-sum-poly-mv}
    \V(t_1 K_1 + \cdots + t_n K_n) = \sum_{i_1,\dots,i_n = 1}^n t_{i_1} \cdots t_{i_n} \MV(K_{i_1}, \dots, K_{i_n}).
\end{align}
The real number $\MV(K_1,\dots,K_n)$ is called the mixed volume of $K_1,\dots,K_n$.
\end{definition}

We use the following known properties of mixed volumes.

\begin{theorem}[\cite{hug2020lectures} Theorem 3.9] \label{thm: mv-properties}
    Let $K_1,\dots,K_n \in \mathcal K^n$. Then:
    \begin{enumerate}
        \item $\MV(K_1,\dots,K_n) \geq 0.$
        \item If $Q_1,\dots,Q_n \in \mathcal K^n$ are such that $K_j \subseteq Q_j$ for each $j$, then 
        $$\MV(K_1,\dots,K_n) \leq \MV(Q_1,\dots,Q_n).$$
        \item For any permutation $\sigma$ of $\{1,\dots,n\}$, 
        $$\MV(K_1,\dots,K_n) = \MV(K_{\sigma(1)}, \dots, K_{\sigma(n)}).$$
        
        \item If $j \in \{1,\dots,n\}$ and $Q_j \in \mathcal K^n$, then \begin{align*}
            \MV(K_1,\dots, K_j + Q_j,& \dots, K_n)
            = \MV(K_1,\dots, K_n) + \MV(K_1,\dots,Q_j,\dots,K_n).
        \end{align*}
        \item For any $t \geq 0$, 
        $$\MV(K_1,\dots,tK_j,\dots,K_n) =  \MV(K_1,\dots,K_n)t.$$
        \item For $x_1, \dots, x_n \in \R^n$, $$\MV(K_1+x_1,\dots, K_n+x_n) = \MV(K_1,\dots, K_n).$$
    \end{enumerate}
\end{theorem}

We will make use of the following alternative notations for mixed volume. If $m \geq 1$ and $K_1,\dots,K_m \in \mathcal K^n$, we will write
$$\MV(K_1[\alpha_1], \dots, K_m[\alpha_m]) := \MV(\underbrace{K_1,\dots,K_1}_{\alpha_1\textup{ times}}, \dots, \underbrace{K_m,\dots,K_m}_{\alpha_m\textup{ times}}),$$
for non-negative integers $\alpha_1,\dots,\alpha_m$ with $\alpha_1 + \cdots + \alpha_m = n$. As mentioned in Section ~\ref{sec: Introduction}, if $K,L \in \mathcal K^n$, we will simply write
\begin{equation}\label{eq: mixed_volume}
\MV_j(K,L) := \MV(K[n - j], L[j]).
\end{equation}
We recall and formally define $d_G$ as presented in Theorem \ref{thm: Main theorem}.

\begin{definition}\label{def: quasi-metric-defn} For a convex body $G \subset \R^n$, let $d_G: \mathcal{K}(G) \times \mathcal{K}(G) \to \R$ be the quasi-metric
    \begin{equation}\label{eqn: dg-def}
    d_G(K,L) := \sum_{j=1}^n\left(2\MV_j(G, K\hull L)-\MV_j(G,K)-\MV_j(G,L)\right).
    \end{equation}
\end{definition}

We define an alternative, related quasi-metric on $\mathcal K(G)$. 
\begin{definition} \label{def: rhoG-deltaG} Letting $G$, $K$, and $L$ be as in Definition \ref{def: quasi-metric-defn}, let 
    \begin{equation}\label{eqn: rhoG-def}
        \rho_G(K,L) := 2\V(G+K\hull L)-\V(G+K)-\V(G+L)
    \end{equation}
\end{definition}

That $\rho_G$ is a quasi-metric follows from the fact that $\rho_G$ and $d_G$ are quasi-isometric. More specifically, 
\begin{equation}\label{eqn: rhoG-dG-qisometry}
    d_G(K,L) \leq \rho_G(K,L) \leq \binom{n}{\lfloor n/2 \rfloor} d_G(K,L).
\end{equation}
Indeed, by Equation \eqref{eq: minkow-sum-poly-mv}, it follows \begin{align*}
    \rho_G(K,L) &= \sum_{j=1}^n\binom{n}{j}\left(2\MV_j(G,K\hull L)-\MV_j(G,K)-\MV_j(G,L)\right)\\
    &\geq \sum_{j=1}^n 2\MV_j(G,K\hull L)-\MV_j(G,K)-\MV_j(G,L)\\
    &=d_G(K,L).
\end{align*}
Similarly, \begin{align*}
    \rho_G(K,L) &= \sum_{j=1}^n \binom{n}{j} \left(2\MV_j(G,K\hull L)-\MV_j(G,K)-\MV_j(G,L)\right)\\
    &\leq \binom{n}{\lfloor n/2\rfloor} \sum_{j=1}^n 2\MV_j(G,K\hull L)-\MV_j(G,K)-\MV_j(G,L)\\
    &= \binom{n}{\lfloor n/2\rfloor}d_G(K,L),
\end{align*}
which yields Equation \ref{eqn: rhoG-dG-qisometry}.

\section{H\"older Estimates for General Bodies and Smooth Bodies} \label{sec: holder estimates}

In this section, we prove \ref{thm: Main theorem}. To do this, we prove Theorem \ref{thm: lipschitz-upper} and Theorem \ref{thm: general-holder-lower}, which, when combined, gives Theorem \ref{thm: Main theorem}. The addition of Theorem \ref{thm: c2-holder-lower} gives the refinement of Theorem \ref{thm: improvement for smooth bodies}.

\begin{theorem}\label{thm: lipschitz-upper}
    There exists $C > 0$ such that \begin{equation}\label{eqn: Lipschitz-Upper}
        d_G(K,L) \leq C \cdot d_\H(K,L)
    \end{equation}
    for all $K,L \in \mathcal K(G)$. 
\end{theorem}

\begin{proof}
    Using the linearity and monotonicity of mixed volume (Theorem \ref{thm: mv-properties}(2,4,5)) with the definition of Hausdorff distance (\ref{def: hausdorff-dist}) shows that for sets $U,V,K_1,\dots,K_{n-1} \in \mathcal K(G)$,
    $$|\MV(K_1,\dots,K_{n-1}, U) - \MV(K_1,\dots,K_{n-1},V)| \leq \MV(K_1,\dots,K_{n-1},B)d_\H(U,V).$$
    (See \cite{hug2020lectures}, Lemma 4.1.)
    Since $K_1,\dots,K_{n-1} \subset G$, using the notation of \eqref{eq: mixed_volume} we get that
    $$|\MV(K_1,\dots,K_{n-1}, U) - \MV(K_1,\dots,K_{n-1},V)| \leq \MV_1(G,B) d_\H(U,V).$$
    Therefore, since $K,L, K\hull L \subset G$, we have
    \begin{align*}
        &\MV_j(G, K\hull L) - \MV_j(G, K) \\
        &=\sum_{\ell=0}^{j-1} \MV(G[n-j], K\hull L[\ell + 1], K[j - \ell - 1]) - \MV(G[n-j], K\hull L[\ell], K[j - \ell])\\
        &\leq \sum_{\ell=0}^{j-1} \MV_1(G, B)d_\H(K\hull L, K)\\
        &\leq j\MV_1(G, B)d_\H(K,L).
    \end{align*}
   The final inequality comes from the fact $d_\H(K\hull L, K) \leq d_\H(K,L)$. The same inequalities are true if the roles of $K$ and $L$ are switched, and so 
    \begin{align*}
        d_G(K,L) &= \sum_{j=1}^n\left[2MV_j(G, K\hull L)-MV_j(G,K)-MV_j(G,L)\right]\\
        &\leq \sum_{j=1}^n 2j\MV_1(G,B)d_\H(K,L)\\
        &= n(n+1)\MV_1(G,B)d_\H(K,L).
    \end{align*}
    This proves Theorem \ref{eqn: Lipschitz-Upper} with $C = n(n+1)\MV_1(G,B)$.
\end{proof}

We next lower bound $d_G$ by $d_\H^n$.

\begin{theorem}\label{thm: general-holder-lower}
    There exists $C > 0$ such that 
    $$C \cdot d_\H(K,L)^n \leq d_G(K,L)$$
    for all $K,L \in \mathcal K(G)$.
\end{theorem}

To prove Theorem \ref{thm: general-holder-lower}, we require two lemmas. Below is the first.

\begin{lemma}\label{lemma: line-segment-halfspace}
    Let $K,L \in \mathcal K^n$. There exists $p \in K$ and $q\in L$ such that $|q - p| = d_\H(K,L)$ and either $(p,q] \subset K\,\hull\,L \setminus K$ or $[p,q) \subset K\,\hull\, L \setminus L$. Moreover, there exists a hyperplane orthogonal to $[p,q]$ running through $p$ (resp. $q$) which supports $K$ (resp. $L$).
\end{lemma}

\begin{proof}
    Recalling Definition \ref{def: hausdorff-dist}, without loss of generality, we assume
    $$d_\H(K,L) = \sup_{x\in L}d(x,K).$$ 
    The compactness of $K$ and $L$ yields $p \in K$ and $q\in L$ such that $d_\H(K,L) = |q - p|$. We claim that $(p,q] \subset K\,\hull\, L \setminus K$. Clearly, $(p,q] \subset K\,\hull\, L$. To show $(p,q] \subset K^c$, we argue by contradiction. If $(p,q] \not\subset K^c$, then there is $t\in (0,1]$ with $tq + (1 - t)p \in K$. It follows $$|q - (tq + (1 - t)p)| = (1-t)|q-p| < |p - q|,$$ contradicting the choice of $p$ as the closest point in $K$ from $q$. 

    Finally, let $H: \R^n \to \R$ be $$H(x) = \frac{q-p}{|q-p|} \cdot (x-p).$$ Clearly, $\{H=0\}$ is a hyperplane running through $p$ and orthogonal to $[p,q]$. By \cite[Theorem 1.14]{hug2020lectures} applied with $q=x$, $K=A$, $\{H=0\}$ supports $K$ at $p$, and $K \subset \{H \leq 0\}$.
\end{proof}

To preface our second lemma, we define a cap of a convex body.

\begin{definition}\label{defn: cap}
    Let $u \in \mathbb{S}^{n-1}$, $h >0$, and $G \subset \R^n$ be a convex body. Recalling the definition of support functions from \ref{def: support-function}, the cap of $G$ in the direction $u$ with height $h$ is given by $$C_G(u,h):= \{x\in G:u\cdot x\geq h_G(u)-h\}.$$
\end{definition}

In lay terms, a cap associated to a convex body $G$ cuts away part of $G$ using a supporting hyperplane, with $h$ being the height of the cap.

\begin{lemma}\label{lemma: cap-existence}
    Let $K, L \in \K(G)$. There exists $u \in \mathbb{S}^{n-1}$ and $q \in L$ (or $q\in K$) such that $$\mathrm{int}( C_G(u, d_\H(K,L))+q) \subset (G+K\hull L) \setminus(G+K)$$ (or respectively $\subset (G+K\hull L) \setminus(G+L)$).
\end{lemma}

\begin{proof}
    Label $d := d_\H (K,L)$. By Lemma \ref{lemma: line-segment-halfspace}, without loss of generality, there exists $p \in K$ and $q \in L$ such that $\abs{q-p}=d$ and $(p,q] \subset K \hull L \setminus K$. Let $u := \frac{q-p}{d}$ and  $H: \R^n \to \R$ be the linear map 
    $$H(x)=u\cdot(x-p).$$ 
    By Lemma \ref{lemma: line-segment-halfspace} again, $\{H\leq0\}$ is a supporting half-space of $K$ at $p$, which implies $G+K \subset G+\{H \leq 0\}$. Additionally, $G+\{H \leq 0\} \subset \{H \leq h_G(u)\}$, since given $x\in G$ and $y \in \{H \leq 0\}$, \begin{align*}
        H(x+y) &= u\cdot (x+y - p)= u\cdot x + u\cdot (y - p) \leq h_G(u).
    \end{align*} Since $q \in K\hull L$, it follows $(G+q) \subset (G+K\hull L)$, which implies $$((G+q) \cap \{H>h_G(u)\}) \subset ((G+K\hull L) \setminus (G+K)).$$
    Finally, because $H(q)=u\cdot(q-p)=d$,
    \begin{align*}
        (G+q)\cap\{H>h_G(u)\}&=G\cap\big(\{H>h_G(u)\}-q\big)+q\\
        &=G\cap\{H>h_G(u)-d\}+q,
    \end{align*}
    which contains
    $$\mathrm{int}(C_G(u,d)+q)=\mathrm{int}(G\cap\{H\geq h_G(u)-d\}+q).$$
\end{proof}

We are now finally ready to prove Theorem \ref{thm: general-holder-lower}.

\begin{proof}[Proof of Theorem \ref{thm: general-holder-lower}]
    Let $d := d_\H(K,L)$. By Equation \ref{eqn: rhoG-dG-qisometry}, it suffices to find $C > 0$ such that $$C \cdot d_\H(K,L)^n \leq \rho_G(K,L) = 2\V(G+K\hull L) - \V(G+K)-\V(G+L)$$ for all $K, L \in \K(G)$. It then again suffices to lower bound either of$$\V(G+K\hull L)-\V(G+K)\quad\text{or}\quad\V(G+K\hull L)-\V(G+L).$$ By Lemma \ref{lemma: cap-existence}, there exists a translate of $\mathrm{int}(C_G(u, d))$ in one of these two sets. Without loss of generality, let $u\in\mathbb{S}^{n-1}$ and $q\in L$ such that $$\mathrm{int}(C_G(u, d)))+q\subset ((G+K\hull L)\setminus (G+K)).$$
    As a direct consequence,
\begin{equation}\label{eq: vol_est}
\V(C_G(u,d))\leq \V(G+K\hull L)-\V(G+K).
\end{equation}
   Next we claim that there exists a cone with base area $\alpha \cdot (\beta d)^{n-1}$ and height $\beta d$ contained in $C_G(u,d)$, where $\alpha, \beta > 0$ are some constants depending only on $G$.
    
    Firstly, assume without loss of generality by translation that $0 \in \mathrm{int}\,G$. There then exists $r>0$ such that $rB \subset G$. Recalling Definition \ref{def: support-function},
    since $G$ is compact, there exists some $x_0\in G$ such that $u\cdot x_0 = h_G(u)$. Define the set
    $$U := (rB \cap \{x: u \cdot x =0\})\,\hull\,\{x_0\}.$$
    $U$ is thus a cone contained in $G$ with height $h_G(u)$ and an $(n-1)$-spherical base of radius $r$. We now define the set
    $$V_0 := C_G(u,d) \cap U = \{u\cdot x \geq h_G(u)-d\} \cap U.$$
Note that $u\cdot x_0>h_G(u)-d$, and $\{u\cdot x = h_G(u)-d\}$ is parallel to the base of $U$. As a result, using homotheties, $V_0$ is a cone with height $\min\{d, h_G(u)\}$ and an $(n-1)$-spherical base of radius $\min\{\frac{rd}{h_G(u)}, r\}$. Furthermore, since $d \leq \mathrm{diam}(G)$ and $r\leq h_G(u)\leq\mathrm{diam}(G)$, 
    $$\frac{rd}{\mathrm{diam}(G)} \leq \min\{d,h_G(u)\},\quad\text{and}\quad\frac{r^2d}{\mathrm{diam}(G)^2} \leq \min\Big\{\frac{rd}{h_G(u)},r\Big\}.$$
    This implies there is a cone $V_1\subset V_0$ with height $\frac{rd}{\mathrm{diam}(G)}$ and an $(n-1)$-spherical base of radius $\frac{r^2d}{\mathrm{diam}(G)^2}$, proving the claim. 

Using \eqref{eq: vol_est}, since $C_G(u,d)$ contains $V_1$, we obtain that
    $$\V(C_G(u,d))\geq\V(V_1)=\frac{\omega_{n-1}}{n}\left(\frac{r}{\mathrm{diam}(G)}\right)^{2n-1}d^n.$$
    Taking $C = \frac{\omega_{n-1}}{n}\left(\frac{r}{\mathrm{diam}(G)}\right)^{2n-1}$ gives the theorem.
\end{proof}

In fact, a stronger result than Theorem \ref{thm: general-holder-lower} can be shown when we impose a smoothness assumption on the boundary of the convex body $G$.

\begin{definition}\label{defn: rolls freely}
We say a ball of radius $r\geq0$ \textit{rolls freely in} $G$ if, at each point $p$ on the boundary of $G$, there is a ball of radius $r$ containing $p$ and contained in $G$, i.e. there is $x\in\R^n$ such that $p\in rB+x\subset G$.
\end{definition}

\begin{remark}
The property that a ball of fixed radius $r>0$ rolls freely in $G$ is equivalent to many other boundary regularity properties of $G$. For example, in \cite{schneider2014}, Theorem 3.2.2 gives that this is equivalent to $G$ being decomposable with $rB$ as a summand, for some $r>0$, and \cite{walther1999} shows this is equivalent to the property that $\partial G$ is a $C^1$ embedded manifold whose spherical image map $\nu:\partial G\to\mathbb{S}^{n-1}$, sending boundary points to their outer unit normal vector, is Lipschitz with constant less than or equal to $\frac{1}{r}$. In particular, if $\partial G$ is $C^2$ then there exists a fixed $r>0$ such that a ball of radius $r$ rolls freely in $G$.
\end{remark}

\begin{theorem}\label{thm: c2-holder-lower}
If a ball of radius $r>0$ rolls freely in $G\in\K^n$, then there exists $C>0$ such that
$$C\cdot d_\H(K,L)^{\frac{n+1}{2}}\leq d_G(K,L)$$
holds for all $K,L\in\K(G)$. In particular, this holds when $G$ is a $C^2$ convex body.
\end{theorem}

To prove this, we compute the volume of an $n$-dimensional spherical cap. The calculation is standard, but since we could not find an appropriate reference, we include the details.

\begin{lemma}\label{lemma: cap-vol-lemma}
    For any $u \in \mathbb S^{n-1}$ and $0 \leq h \leq r$,
    $$\V(C_{rB}(u,h)) = \omega_{n-1}r^n\int_0^{\arccos(1-h/r)}(\sin\theta)^n\,d\theta.$$
\end{lemma}

\begin{proof}
    By rotational symmetry, $\V(C_{rB}(u,h))$ does not depend on the choice of $u$, and so we can assume that $u = e_n$ is the $n$-th standard basis vector. Since the slice of a spherical cap is a ($n-1$)-dimensional sphere, Fubini's theorem shows that
    \begin{align*}
        \V(C_{rB}(e_n,h)) &= \int_0^h \V_{n-1}(\{x \in \R^{n-1}: |x|^2 \leq r^2 - (r - t)^2\})\,dt\\
        &= \omega_{n-1}\int_0^h (r^2 - (r - t)^2)^{\frac{n-1}{2}}\,dt,
    \end{align*}
    where $\V_{n-1}$ is the ($n-1$)-dimensional Lebesgue measure. Making the appropriate substitutions, 
    $$\V(C_{rB}(e_n,h)) = \omega_{n-1} r^n \int_{1-h/r}^1 (1 - x^2)^{\frac{n - 1}{2}}\,dx,$$
    which reduces to the stated formula by substituting $x = \cos\theta$.
\end{proof}

We now prove \ref{thm: c2-holder-lower}.

\begin{proof}
Without loss of generality, we can translate $G$ so that $0 \in \text{int}(G)$. Let $\nu:\partial G\to\mathbb{S}^{n-1}$ again be the spherical image map sending boundary points to their outer unit normal vector. 

By Lemma \ref{lemma: cap-existence}, it suffices to show that there exists $C'>0$ such that, for all $u\in\mathbb{S}^{n-1}$ and $0<h<\text{diam}(G)$,
\begin{equation}C'h^{\frac{n+1}{2}}\leq \V(C_G(u,h)),\label{eqn: theorem 3.6 goal}\end{equation}
where $C_G(u,h)$ denotes the cap of $G$ in the direction $u$ with height $h$ (see Definition \ref{defn: cap}).
This would then imply, for $K,L\in\K(G)$, with $u\in\mathbb{S}^{n-1}$ selected as in Lemma \ref{lemma: cap-existence},
\begin{align*}
C'd_\H(K,L)^{\frac{n+1}{2}}&\leq \V(C_G(u,d_\H(K,L)))\leq\V(G+K\,\hull\, L)-\V(G+K)\\
&\leq\binom{n}{\lfloor n/2\rfloor}d_G(K,L),
\end{align*}
where the last inequality is due to inequality \ref{eqn: rhoG-dG-qisometry}.

Now let $r>0$ be such that a ball of radius $r$ rolls freely in $G$. That is, for each $p\in\partial G$ there is some $x\in\R^n$ such that $p\in rB+x\subset G$.
Now, fixing any $p\in\partial G$, and choosing the appropriate such $x\in\R^n$, the spherical cap $C_{rB+x}(\nu(p),h)$ is contained in the cap $C_G(\nu(p),h)$, and thus
$$\V(C_G(\nu(p),h))\geq\V(C_{rB+x}(\nu(p),h)).$$
Supposing $h\leq r$, by \ref{lemma: cap-vol-lemma}, the volume of a radius-$r$ spherical cap of height $h$ is given by 
\begin{equation*}
\omega_{n-1}r^n\int_0^{\arccos(1-h/r)}(\sin\theta)^n\,\mathrm{d}\theta.
\end{equation*}
From the estimates $\sin\theta\geq\frac{2}{\pi}\theta$ for $0\leq\theta\leq\frac{\pi}{2}$, which holds for $h\leq r$, and $\arccos(1-t)\geq\sqrt{2t}$, we obtain
\begin{align*}
\V(C_G(\nu(p),h))&\geq\omega_{n-1}\Big(\frac{2r}{\pi}\Big)^n\int_0^{\sqrt{2h/r}}\theta^n\,\mathrm{d}\theta=\frac{\omega_{n-1}\pi^{n+1}}{n+1}\left(\frac{2r}{\pi}\right)^n\left(\frac{2h}{r}\right)^{\frac{n+1}{2}}.
\end{align*}
Since this holds for any choice of $p\in\partial G$ and $0<h\leq r$, if we take $$C'=\frac{\omega_{n-1}\pi^{n+1}r^{(n-1)/2}}{2^{(n+1)/2}(n+1)},$$
we get the desired inequality \ref{eqn: theorem 3.6 goal} in the case $h\leq r$. If $h\geq r$, then a hemisphere of $rB+x$ is contained in $C_G(\nu(p),h)$, and so, since $h\leq\text{diam}(G)$, we simply have
$$\V(C_G(\nu(p),h))\geq\frac{1}{2}\V(rB)\geq\frac{1}{2}\V(rB)\left(\frac{h}{\text{diam}(G)}\right)^\frac{n+1}{2}.$$
Thus taking $C'=\frac{1}{2}\omega_nr^n\text{diam}(G)^{-(n+1)/2}$ gives inequality \ref{eqn: theorem 3.6 goal} in the case where $h\geq r$.
\end{proof}

\section{H\"older Exponent Optimality} \label{sec: holder exponent optimality}

In the section, we prove Theorem \ref{thm: bounds are optimal}. To do this, we prove Theorems \ref{thm: lipschitz-optimality}, \ref{thm: polytope-optimality}, and \ref{thm: n+1/2 optimal}, which combined constitute Theorem \ref{thm: bounds are optimal}.

\begin{theorem}\label{thm: lipschitz-optimality}
     No inequality of the form $d_G(K,L) \leq C \cdot d_\H(K,L)^\beta$ holds for all $K,L \in \mathcal K(G)$ when $\beta > 1$, for any choice of $C > 0$. 
\end{theorem}

\begin{proof}
    By way of contradiction, suppose that there is $\beta > 1$ and $C>0$ such that $d_G(K,L) \leq C \cdot d_\H(K,L)^\beta$ for all $K,L \in \K(G)$. Equivalently,
    $$\frac{d_G(K,L)}{d_\H(K,L)^\beta}$$ is uniformly bounded for all $K \neq L$, which is again equivalent to
    $$\frac{\rho_G(K,L)}{d_\H(K,L)^\beta}$$
    being uniformly bounded for $K \neq L$, recalling Definition \ref{def: rhoG-deltaG} and Equation \eqref{eqn: rhoG-dG-qisometry}. By the translational invariance of the Lebesgue measure, we can assume that $0 \in G$, so that $rG \subset G$ for $0 < r < 1$. Then $$\rho_G(G, rG) = \V(G + G) - \V(G + rG) = \V(G)(2^n - (1 + r)^n)$$ and $d_\H(G, rG) = M(1 - r)$, where $M = \max_{|x| = 1} |h_G(x)|$. We deduce
    $$\frac{\rho_G(G,rG)}{d_\H(G,rG)^\beta} = \frac{\V(G)}{M^\beta} \frac{2^n - (1 + r)^n}{(1 - r)^\beta}$$
    is uniformly bounded for all $r \in (0,1)$, which is a contradiction since $\frac{2^n - (1 + r)^n}{(1 - r)^\beta} \to \infty$ as $r$ increases to $1$.    
\end{proof}

The following lemma will allow us, by constructing examples which exhibit the worst-case behavior, to prove optimality results on the exponents of H\"older bounds of $d_G$ by $d_\H$.

\begin{lemma}\label{lemm: cap-counterexample}
Let $G\subset\R^n$ be a convex body, and let $u\in\mathbb{S}^{n-1}$. Let, for $t>0$,
\begin{equation}C_t=C_G(u,t),\quad\text{and}\quad S_t=G\cap\{x\in\R^n:u\cdot x=h_G(u)-t\}.
\label{eqn: cap used for counterexample}\end{equation}
Then the following hold:
\begin{enumerate}
\item $d_\H(C_t,S_t)\geq t$.
\item $d_\H(C_t,S_t)\to0$ as $t\to0$.
\item $\V(G+C_t)-\V(G+S_t)\leq 2^n\V(C_t)$.
\end{enumerate}
\end{lemma}
\begin{proof}
Statement (1) follows from the fact that $S_t\subset C_t$, and also $S_t$ is contained in the hyperplane $\{x\in\R^n:u\cdot x=h_G(u)-t\}$. Therefore,
$$d_\H(C_t,S_t)=\sup_{p\in C_t}d(p,S_t)\geq\sup_{p\in C_t}|u\cdot p-(h_G(u)-t)|=t.$$

To prove (2), we claim that $C_t$ and $S_t$ converge to $C_0:=\{x\in G:u\cdot x=h_G(u)\}$, which is the facet of $G$ with $u$ as an outward normal vector. Firstly, note that $\{C_{1/n}\}_{n=1}^\infty$ is a decreasing family of compact sets whose intersection is $C_0$, and so $C_{1/n}\to_{d_\H}C_0$ (see, for instance, Lemma 1.8.2 of \cite{schneider2014}). Then for any $\delta>0$, we can choose $n$ large enough that for all $t<1/n$,
\begin{equation}\label{eq: contain}
S_t\subset C_t\subset C_{1/n}\subset C_0+\delta B.
\end{equation}
Let $b \in G$. We have that $u \cdot b = h_G(u) - \varepsilon$. Then for $t < \varepsilon$ we have that 

$$\frac{t }{ \varepsilon}  b +  \Big(1-\frac{t }{ \varepsilon} \Big) C_0 \subset S_t.$$
Since $C_0$ is compact, for all $t$ small enough we have that 

$$C_0  \subset \frac{t }{ \varepsilon}  b +  \Big(1-\frac{t }{ \varepsilon} \Big) C_0  + \delta B \subset S_t + \delta B.$$ 

This containment together with \eqref{eq: contain} implies that $d_\mathcal H(S_t,C_0) \to 0$ as $t \to 0$, finishing the proof of (2).

Lastly, we show statement (3) by showing $(G+C_t)\setminus(G+S_t)\subset 2C_t$. This is equivalent to (3) because $S_t\subset C_t$. Let $x=p+q$, where $p\in G$ and $q\in C_t$, and assume $x\notin G+S_t$. We proceed in two cases:

\smallskip
\noindent \textbf{Case 1: $p\in C_t$.} Then $x=p+q \in 2 C_t.$

\smallskip
\noindent \textbf{Case 2: $p\notin C_t$.} Then there must be some $\lambda\in[0,1]$ such that
$$u\cdot(\lambda p+(1-\lambda)q)=h_G(u)-t,$$
since $u\cdot p<h_G(u)-t$ and $u\cdot q\geq h_G(u)-t$. But again by convexity,
$p':=\lambda p+(1-\lambda)q$ and $q':=(1-\lambda)p+\lambda q$
are both contained in $G$. Then, by construction, $p'\in S_t$, and so $x = q' +p' \in G+S_t$, a contradiction.
\end{proof}

We now show that the $n$ is the optimal H\"older exponent for the lower H\"older bound of $d_\H$ on $d_G$ for $G$ being a polytope.

\begin{theorem}\label{thm: polytope-optimality}
If $G$ is a convex polytope, and  $p<n$, there is no $C>0$ such that 
$$C\cdot d_\H(K, L)^p \leq d_G(K, L)$$ holds for all $K,L\in\K(G)$.
That is, the exponent $n$ gives the optimal H\"older bound for $G$ a polytope. 

\end{theorem}

\begin{proof}

Without loss of generality, we translate $G$ such that $0 \in \text{int}(G)$. 

Since $G$ is compact, we can choose $x \in G$ which maximizes $|x|$, which must then be a vertex of $G$ by the structure of convex polytopes. Let $u = x/|x|$; recalling Definition \ref{def: support-function}, we then know that $h_G(u)$ is achieved by $u\cdot  x$. 

For any $t \geq 0$, we take $C_t$ and $S_t$ as constructed in Lemma \ref{lemm: cap-counterexample} using normal vector $u$. Since $G$ is a convex polytope and thus has a finite number of vertices, for all $t$ sufficiently small, $C_t$ is an $n$-dimensional pyramid with base $S_t$ and height $t$; more specifically, $C_t = x \, \hull\, S_t$. As a result, we know $\V(C_t) = At^n$ for some constant $A>0$. Consequently, by Lemma \ref{lemm: cap-counterexample} parts (1) and (3) and Equation \ref{eqn: rhoG-dG-qisometry}, \begin{align*}
    d_\H(C_t, S_t)^n \geq t^n &= \frac{1}{A}\V(C_t)\geq  \frac{1}{2^nA} \left(\V(G+C_t)-\V(G+S_t)\right)\\
    &=\frac{1}{2^nA} \rho_G(C_t, S_t)\geq \frac{1}{2^nA}d_G(C_t, S_t),
\end{align*}
for all $t$ sufficiently small. 

Assume now for contradiction that there exists some $C >0$ and $p <n$ such that $$C\cdot d_\H(K, L)^p \leq d_G(K, L)$$ holds for all $K,L\in\K(G)$. By (2) of Lemma \ref{lemm: cap-counterexample}, we then have some $C'>0$ such that $$C' \cdot d_\H(C_t, S_t)^p \leq d_\H(C_t, S_t)^n$$ for all $t$ sufficiently small, which is clearly untrue.

\end{proof}

Finally, we show that for a general convex body $G$, $\frac{n+1}{2}$ is the lower bound on the set of H\"older exponents for the lower H\"older bound of $d_\H$ on $d_G$. 

\begin{theorem}\label{thm: n+1/2 optimal}
Let $G\subset\R^n$ be any convex body. Then if $p<\frac{n+1}{2}$, there is no $C>0$ such that
\begin{equation}C\cdot d_\H(K,L)^p\leq d_G(K,L)\label{eqn: n+1/2 optimality}\end{equation}
holds for all $K,L\in\K(G)$.
\end{theorem}
\begin{proof}
Similar to the proof of Theorem \ref{thm: polytope-optimality}, we will show that $C_t$ and $S_t$ as defined in Lemma \ref{lemm: cap-counterexample} are a family of sets that serve as a contradiction to the inequality \ref{eqn: n+1/2 optimality}.

Let $x\in G$ be such that $|x|\geq|y|$ for all $y\in G$ (such a point exists by compactness of $G$). Let $u \in \mathbb{S}^{n-1}$ and $r > 0$ be such that $x=ru = h_G(u)u$. Then the hyperplane
$$H=\{y\in\R^n:u\cdot y=r\}$$
supports $rB$ at $ru$, but also supports $G$ at $ru$, since $ru\in H\cap G$ and $G\subset rB$.
Now, for $t\geq0$, consider $C_t$ and $S_t$ as constructed in Lemma \ref{lemm: cap-counterexample} with our $u$ and $G$. By Lemma \ref{lemm: cap-counterexample} and Equation \ref{eqn: rhoG-dG-qisometry}, for all $t\geq0$,
$$d_G(C_t,S_t)\leq\V(G+C_t)-\V(G+S_t)\leq 2^n\V(C_t).$$
But additionally, for all $t\geq0$,
$$C_t\subset C_{rB}(u, t).$$
The set on the right hand side is a spherical cap of radius $r$ and height $t$. The volume of such a cap is, by Lemma \ref{lemma: cap-vol-lemma},
\begin{equation*}
\V(C_{rB}(u, t)) =\omega_{n-1}r^n\int_0^{\arccos(1-t/r)}(\sin\theta)^n\,\mathrm{d}\theta.
\end{equation*}
From the estimates $\sin\theta\leq\theta$ and $\arccos(1-x)\leq\pi\sqrt{x/2}$ for $0\leq x\leq1$, we obtain
\begin{align*}
d_G(C_t,S_t)&\leq2^n\V(C_t)\leq\omega_{n-1}r^n\int_0^{\pi\sqrt{t/2r}}\theta^n\,\mathrm{d}\theta=\frac{\omega_{n-1}\pi^{n+1}r^n}{n+1}\left(\frac{t}{2r}\right)^{\frac{n+1}{2}}.
\end{align*}
By (1) in Lemma \ref{lemm: cap-counterexample},
it follows 
\begin{equation}\label{eqn: n+1/2 upper}
    d_G(C_t, S_t) \leq C'd_\H(C_t, S_t)^{\frac{n+1}{2}}
\end{equation}
where $C' = \frac{\omega_{n-1}\pi^{n+1}}{n+1}(2r)^{\frac{n-1}{2}} > 0$.
Combining (2) of Lemma \ref{lemm: cap-counterexample} with Equation \ref{eqn: n+1/2 upper}, we deduce that the inequality 
$$C\cdot d_\H(K,L)^p\leq d_G(K,L)$$
does not hold for for all $K$ and $L$ for any choice of $C > 0$ when $p<\frac{n+1}{2}$. 
\end{proof}

\bibliographystyle{abbrv}
\bibliography{bib}

\noindent{\sc A. Aitokhuehi, Rice University}\\
{\tt ayo.aitokhuehi@rice.edu}\vspace{0.1in}

\noindent{\sc B. Braiman, University of Wisconsin-Madison}\\
{\tt braiman@wisc.edu}\vspace{0.1in}

\noindent{\sc D. Cutler, Tufts University}\\
{\tt david.cutler@tufts.edu}\vspace{0.1in}

\noindent{\sc T. Darvas, P. Gupta, V. Pidaparathy, University of Maryland}\\
{\tt tdarvas@umd.edu, pgupta8@umd.edu, pvasanth@umd.edu }\vspace{0.1in}

\noindent{\sc R. Deaton, Georgia Institute of Technology}\\
{\tt rdeaton9@gatech.edu}\vspace{0.1in}

\noindent{\sc J. Horsley, University of Utah}\\
{\tt u0981992@utah.edu}\vspace{0.1in}

\noindent{\sc J. Tang, University of Chicago}\\
{\tt t12821@uchicago.edu}\vspace{0.1in}

\end{document}